\documentclass[11pt]{article} 
\usepackage{amsmath}
\usepackage{amsfonts}
\usepackage{amssymb}
\usepackage{amsthm}
\usepackage{dsfont}
\usepackage{graphicx}

\theoremstyle{plain}
\newtheorem{theorem}{Theorem}[section]
\newtheorem{corollary}[theorem]{Corollary}
\newtheorem{lemma}[theorem]{Lemma}
\theoremstyle{definition}
\newtheorem{definition}[theorem]{Definition}
\newtheorem{question}[theorem]{Question}
\newtheorem{remark}[theorem]{Remark}
\newtheorem{example}[theorem]{Example}

\newcommand{\GId}{I\!\!I}

\newcommand{\bV}{{\mathbb{V}}}
\newcommand{\GA}{{\mathbb{A}}}

\newcommand{\GP}{{\mathbb{P}}}
\newcommand{\GN}{{\mathbb{N}}}
\newcommand{\GR}{{\mathbb{R}}}

\newcommand{\GD}{\mathds{D}}
\newcommand{\GV}{\mathds{V}}

\newcommand{\cA}{{\mathcal A}}
\newcommand{\cL}{{\mathcal L}}

\newcommand{\cF}{{\mathcal F}}

\newcommand{\cX}{{\mathcal X}}

\newcommand{\cC}{{\mathcal C}}

\newcommand{\bis}{\prime\prime}
\newcommand{\conver}{\mathop{\longrightarrow}}

\newcommand{\intau}{\ \ \conver_{\tau}\ }

\newcommand{\ines}{\ \ \conver_{S}\ }
\newcommand{\starines}{\stackrel{*}{\ines}}
\newcommand{\intopol}[1]{\ \ \conver_{#1}\ }

\newcommand{\lr}{\longrightarrow}

\title{New characterizations\\[2mm] of the $S$ topology on the Skorokhod space}

\author{Adam Jakubowski\thanks{Supported
    by the Polish NCN grant no. 2012/07/B/ST1/03508}
\\Nicolaus Copernicus University\\
adjakubo@mat.uni.torun.pl}

\begin{document}

\maketitle

\begin{abstract}
The $S$ topology on the Skorokhod space was introduced by the author in 1997 and since then it has proved to be a useful tool in several areas of the theory of stochastic processes. The paper brings complementary information on the $S$ topology. It is shown that the convergence of sequences in the $S$ topology admits a closed form description, exhibiting the locally convex character of the $S$ topology. Morover, it is proved that the $S$ topology is, up to some technicalities, finer than any linear topology which is coarser than Skorokhod's $J_1$ topology. The paper contains also definitions of extensions of the $S$ topology to the Skorokhod space of functions defined on $[0,+\infty)$ and with multidimensional values. 
\\[2mm]
\end{abstract}

\section{Introduction}
The $S$ topology on the Skorokhod space $\GD = \GD([0,T])$ of c\`adl\`ag functions has emerged as a result of chain of observations made in eighties and nineties of the twentieth century. 

In 1984 Meyer and Zheng \cite{MeZh84} considered certain conditions on truncated variations of stochastic processes and proved that these conditions give uniform tightness of the processes in some topology   
(nowadays called the Meyer-Zheng topology) on $\GD$.

A year later Stricker \cite{Str85}  proved that it is possible to relax Meyer-Zheng conditions to uniform tightness of random variables $\{ \|X_n\|_{\infty}\}$ and $\{N^{a,b}(X_n)\}$, for each pair of levels $ a < b$ (where $N^{a,b}$ is the number of up-crossings of levels $a < b$). 

Stricker still operated with the Meyer-Zheng topology.
 It was clear for Kurtz \cite{Kur91} that such conditions give much more. But an {\em ad hoc} device constructed by Kurtz \emph{did not have a topological character}.

The $S$ topology was constructed by the author in \cite{Jak97}. This step was final in the sense that Stricker's conditions are \emph{equivalent} to the uniform tightness in the  $S$ topology.

It should be emphasized, that the $S$ topology, as considered in \cite{Jak97}, is sequential and it is still not known if it is completely regular. So at the moment
of  its creation there was no formalism to deal with it within the Probability Theory. Such formalism has been provided in \cite{Jak2000}. And an efficient tool - the almost sure Skorokhod's representation in non-metric spaces, given in \cite{Jak97asr} - made the $S$ topology an operational and useful device in many problems.

The very first application of the $S$ topology was given by the author in \cite{Jak96}, where convergence of stochastic integrals was considered.
Later the $S$ topology was used in problems related to homogenization of  stochastic differential equations (e.g.  \cite{BEP09}, \cite{BEP17},  \cite{BMPa07}, \cite{Lejay02},  \cite{OuPa02},  \cite{SRP09}, \cite{Rho10}), 
diffusion approximation of solutions to the Poisson equation (\cite{PaVe05}),
stability of solutions to semilinear equations with Dirichlet operator (\cite{Klim17}), martingale transport on the Skorokhod space (\cite{GTT17}),
the Skorokhod problem (\cite{LaSl03}, \cite{LaSl13},  \cite{MRS17}, \cite{Slo15}), econometrics (\cite{ChZh10}),
control theory (\cite{BGM11}, \cite{KuSt01}), linear models with heavy-tails (\cite{BJL2016}), continuity of semilinear Neumann-Dirichlet problems (\cite{MaRa16}), generalized Doob-Meyer decomposition (\cite{Jak06}), 
modeling stochastic reaction networks (\cite{KaKu13})
and even in some considerations of more general character (\cite{CoKu15}, \cite{Ku14}).

In the present paper we provide some complementary information related to the $S$ topology. Section 2 restates the definition and basic properties of the $S$ topology. The $S$ topology was defined in \cite{Jak97} by means of so-called $\cL$-convergence $\ines$, which leads to the  topological convergence $\starines$ via the Kantorovich-Vulih-Pinsker-Kisynski recipe (we refer to Section \ref{appendix} for a primer on sequential spaces). 

In Section \ref{seccompact} we provide a closed form formula for the  $\starines$ convergence.

In Section \ref{sechier} we find a position for $S$ in the hierarchy of topologies on $\GD$, by showing that it is \emph{essentially} finer than any linear topology which is coarser than the Skorokhod's $J_1$ topology.

In Section \ref{extensions} we extend the notion of the $S$ topology to the case of infinite time horizon and to functions with multidimensional values.

\section{Definition of the $S$ topology}

All results in this section are taken from \cite{Jak97}.
We need some standard notation first.
\begin{enumerate}
\item $\GD= \GD([0,T])$ denotes the Skorokhod space, i.e. a family of functions  $x\,:\,[0,T] \to  \GR^1$, which are {right-continuous} at every $t\in [0,T)$ and admit {left-limits} at every $t \in (0,T]$.
\item $\GD$ is naturally equipped with the \emph{sup-norm} $ \|x\|_{\infty} = \sup_{t\in [0,T]} |x(t)|$.
\item For $a < b$, $N^{a,b}(x)$ is \emph{the number of up-crossings of levels $a$ and $b$ by function $x\in \GD$}. In other words, $N^{a,b}(x)$ is the largest integer $k$ such that there are numbers  $0 \leq t_1 < t_2 <t_3 < \ldots < t_{2k-1} < t_{2k} \leq T$ satisfying $x(t_{2i - 1}) < a,\ x(t_{2i}) > b$, for each $i=1,2,\ldots, k$.
\item For $\eta > 0$, $N_{\eta}(x)$ is \emph{the number of $\eta$-oscillations} of $x\in \GD$ on $[0,T]$. This means that $N_{\eta}(x)$ is the largest ineger $k$ such that there are numbers $0 \leq t_1 < t_2 \leq t_3 <  \ldots \leq  t_{2k-1} < t_{2k} \leq T$ satisfying $ \big|x(t_{2i}) - x(t_{2i - 1})\big| > \eta$, for each $i=1,2,\ldots, k$.
\item $\|v\|(T)$ is \emph{the total variation of $v$ on $[0,T]$}:
\[
\|v\|(T)
 = \sup \big\{|v(0)| + \sum_{i=1}^{m} |v(t_i) - v(t_{i-1})|\big\}, \]
where the supremum is taken over all 
$0=t_0 < t_1 < \ldots < t_m = T,\ m\in\GN$.
\item $\bV = \big\{ x \in \GD\,;\, \|x\|(T) < +\infty\big\}$.
\end{enumerate}

The $S$ topology is defined in terms of $S$-convergence. 
\begin{definition}[$S$-convergence]
We shall write $x_n \ines x_0$
if for every $\varepsilon > 0$ one can
find elements $v_{n,\varepsilon}\in \GV$, $n=0,1,2,\ldots $ which are
$\varepsilon$-uniformly close to $x_n$'s and weakly-$*$ convergent:
\begin{eqnarray}
\|x_n - v_{n,\varepsilon}\|_{\infty} \leq \varepsilon,&&\ n = 0, 1, 2,
\ldots, \label{vclose}\\
v_{n,\varepsilon} \Rightarrow v_{0,\varepsilon},&& \text{as $n\to\infty$}.\label{vconv}
\end{eqnarray}
\end{definition}
\begin{remark}
Recall, that $v_{n,\varepsilon} \Rightarrow v_{0,\varepsilon}$ means that 
\begin{equation}\label{evconvclear}
 \int_{[0,T]} f(t) \,dv_{n,\varepsilon}(t) \to \int_{[0,T]} f(t)\,dv_{0,\varepsilon}(t),
\end{equation}
for each continuous function $f : [0,T] \to \GR^1$. In particular, setting $f(t) \equiv 1$ we get
\begin{equation}\label{vone}
v_{n,\varepsilon}(T) \to v_{0,\varepsilon}(T).
\end{equation}
Moreover, by the Banach-Steinhaus theorem, relation (\ref{vconv}) implies
\begin{equation}\label{vbounded}
\sup_n \|v_{n,\varepsilon}\|(T) < +\infty.
\end{equation}
\end{remark}
For the sake of brevity of formulation of the next theorem let us list some conditions describing properties of a subset $K\subset \GD$.
\begin{align}
\label{2e1}
\sup_{x\in K}
\|x\|_{\infty} &< +\infty. \\
\label{2e2}
\sup_{x\in K} N^{a,b}(x) &< +\infty,\quad \text{for all $a < b$}.\\
\label{2e3}
\sup_{x\in K} N_{\eta}(x) &< +\infty, \quad \text{for every $\eta > 0$}.
\end{align}

\begin{theorem} (Criterion of relative $S$-compactness) \label{T:criteria}
Let $K \subset \GD$.
We can find in every sequence $\{x_n\}$ of elements of $K$
 a subsequence $\{x_{n_k}\}$
such that $x_{n_k} \ines x_0$, as $k\to \infty$, if, and only if, 
one of the following equivalent statements is satisfied.
\begin{description}
\item{\bf (i)} Conditions  \ref{2e1} and \ref{2e2} hold.
\item{\bf (ii)} Conditions  \ref{2e1} and \ref{2e3} hold.
\item{\bf (iii)} For every $\varepsilon > 0$ and every $x\in
K$ there exists
$v_{x,\varepsilon} \in \bV$ such that
\begin{equation}\label{2e4}
\sup_{x\in K} \|x - v_{x,\varepsilon}\|_{\infty} \leq \varepsilon,
\end{equation}
and
\begin{equation}\label{2e5}
\sup_{x\in K} \|v_{x,\varepsilon}\| < +\infty.
\end{equation}
\end{description}
\end{theorem}

\begin{definition}[The $S$ topology]
The $S$ topology is the sequential topology on $\GD$ generated by the $\cL$-convergence $\ines$.
\end{definition}

\begin{remark}
``$\ines$'' is not an $\cL^*$-convergence. In the next section we shall give a closed form of ``$\starines$", obtained from $\ines$ by the KVPK recipe (Theorem \ref{kvpk_recipe}).
\end{remark}

\section{Closed form definition of  $\starines$}\label{seccompact}

\begin{definition}\label{gat}
 Let $\GA = \GA\big([0,T]\big)$ be a family of \emph{continuous} functions of \emph{finite variation} ($\GA \subset C([0,T]) \cap \GV$), satisfying {$A(0) = 0$}.
 
Let $A_n \in \GA, n=0,1,2,\ldots$. We will say that $A_n\intopol{\tau} A_0$, if
\begin{equation}
 \sup_{t\in[0,T]} |A_n(t) - A_0(t)| \to 0,\label{aconv}
\end{equation}
and 
\begin{equation} 
\sup_n \|A_n\|(T) < +\infty.\label{asup}
\end{equation}
\end{definition}
\begin{remark}
 This is a {``mixed topology''} on  $C([0,T]) \cap \GV$ .
\end{remark}

\begin{theorem}\label{thmain}
$x_n \starines x_0$ if, and only if, $x_n(T) \to x_0(T)$ and
\begin{equation}\label{eq:int}
 \int_0^T x_n(t)\, dA_n(t) \to \int_0^T x_0(t)\, dA_0(t),
\end{equation}
for each sequence $\{A_n\} \subset \GA$ that satisfies $A_n \intopol{\tau} A_0$.
\end{theorem}

\begin{proof}
Suppose that $x_n \starines x_0$. Then $x_n(T) \to x_0(T)$ is a consequence of (\ref{vone}) and property (\ref{vclose}) of functions $\{v_{n,\varepsilon}\}$. 
In order to prove (\ref{eq:int}) one could observe that this convergence is a very particular (deterministic) case of Theorem 1
(or Theorem 5) in \cite{Jak96}. But it is more instructive to give here a direct proof.

So assume that $A_n \intopol{\tau} A_0$, choose $\varepsilon > 0$ and let $\{v_{n,\varepsilon}\}_{n=0,1,2,\ldots} \subset \GV$ satisfy (\ref{vclose}) and (\ref{vconv}). For $n = 0,1,2,\ldots$ we have
\[  \big| \int_0^T x_n(t)\, dA_n(t) - \int_{[0,T]} v_{n,\varepsilon}(t)\, dA_n(t) \big| \leq  \varepsilon \sup_n \| A_n\|(T),
\]
and therefore it is enough to show that  
\begin{equation}\label{enough}
 \int_{[0,T]} v_{n,\varepsilon}(t)\, dA_n(t)\to  \int_{[0,T]} v_{0,\varepsilon}(t)\, dA_0(t).
\end{equation}
By the integration by parts formula, the continuity of $A_n$ and $A_n(0) = 0$ we obtain that for $n = 0, 1, 2, \ldots $
\begin{align*} \int_{[0,T]} v_{n,\varepsilon}(t)\, dA_n(t) &= v_{n,\varepsilon}(T) A_n(T) - 
\int_{[0,T]} A_n(t) \, dv_{n,\varepsilon}(t)\\
&= v_{n,\varepsilon}(T) A_n(T) - 
\int_{[0,T]} A_0(t) \, dv_{n,\varepsilon}(t) \\
&\qquad\qquad + \int_{[0,T]} \big(A_0(t) - A_n(t)\big) \, dv_{n,\varepsilon}(t) \\
& = I_1(n) + I_2(n) + I_3(n).
\end{align*}
By (\ref{vone}) and (\ref{aconv}) $I_1(n) =  v_{n,\varepsilon}(T) A_n(T) \to
v_{0,\varepsilon}(T) A_0(T) = I_1(0)$. By (\ref{evconvclear}) $I_2(n) = \int_{[0,T]} A_0(t) \, dv_{n,\varepsilon}(t) \to \int_{[0,T]} A_0(t) \, dv_{0,\varepsilon}(t) = I_2(0)$. Finally 
\[ |I_3(n)| =  \big|\int_{[0,T]} \big(A_0(t) - A_n(t)\big) \, dv_{n,\varepsilon}(t)\big| \leq \sup_{t\in [0,T]} \big| A_0(t) - A_n(t)\big| \sup_n \| v_{n,\varepsilon}\|(T) \to 0\]
by (\ref{aconv}) and (\ref{vbounded}). Hence (\ref{enough}) holds.

Now let us assume that $x_n(T) \to x_0(T)$ and (\ref{eq:int}) holds for every sequence $\{A_n\} \subset \GA$, $A_n \intopol{\tau} A_0$. We claim that it is enough to establish relative $S$-compactness of $\{x_n\}$. Indeed, then  in every subsequence $\{n'\}$ we can find a further subsequence $\{n^{\prime\prime}\}$ such that $x_{n^{\prime\prime}} \ines y_0$, for some $y_0 \in \GD$. Take a function $f \in L^1([0,T])$ and define $A_f(t) = \int_0^t f(u)\,du$. Then $A_f \in \GA$ and by the first part of the proof, 
\[ \int_0^T x_{n^{\prime\prime}}(u)\, dA_f(u) \to \int_0^T y_0(u)\, dA_f(u) = \int_0^T x_0(u)\, dA_f(u).\] 
Hence for each integrable $f$ we have 
\[ \int_0^T y_0(u) f(u)\,du =\int_0^T y_0(u)\, dA_f(u)  = \int_0^T x_0(u)\, dA_f(u) = \int_0^T x_0(u) f(u)\,du,\]
and, consequently, $y_0 = x_0$ almost everywhere. Since they are c\`adl\`ag functions, $y_0 = x_0$ on $[0,T)$. And $y_0(T) = x_0(T)$ holds by our assumption.

So far we have proved that in every subsequence $\{n'\}$ we can find a further subsequence $\{n^{\prime\prime}\}$  
along which $x_{n^{\prime\prime}} \ines x_0$. Hence, by the KVPK recipe, 
$x_n \starines x_0$.

In order to prove relative $S$-compactness of $\{x_n\}$ it is necessary to 
adjust integrands $A_n$ in a way suitable for the particular functional determining the relative $S$-compactness via Theorem \ref{T:criteria}.

First let us consider condition (\ref{2e1}). Suppose that $\sup_n \|x_n\|_{\infty} = +\infty.$ Then there exists a subsequence $n_k$ and numbers $t_{k} \in [0,T)$ such that $a_k = |x_{n_k}(t_k)| \to \infty$. Without loss of generality we may assume that $a_k = x_{n_k}(t_k)$ and $t_k < T$. By the right-continuity of $x_{n_k}$ we can find numbers $h_k$ such that $t_k + h_k < T$ and 
\[ x_{n_k}(t) \geq (1/2) a_k,\quad \text{ for } t \in [t_k, t_k + h_k].\]
Let $b_k = \sqrt{a_k} h_k$. Consider function $f_k (u) = (1/b_k) \GId_{[t_k,t_k+h_k]}(u)$ and the corresponding function $A_{f_k} \in \GA$.
 We have 
\begin{align*}
\int_{0}^T x_{n_k}(u) \,dA_{f_k}(u) &= \int_{0}^T x_{n_k}(u) f_k(u) \,du \\
&= 
\frac{1}{b_k} \int_{t_k}^{t_k + h_k} x_{n_k}(u)\,du \geq \frac{a_k h_k}{2 b_k} = (1/2) \sqrt{a_k} \to +\infty,
\end{align*}
while 
\[ \|A_{f_k}\|_{\infty} = \|A_{f_k}\|(T) = \int_0^T f_k(u)\,du = h_k/b_k = 1/\sqrt{a_k} \to 0.\]
It follows that  (\ref{eq:int}) cannot be satisfied.

In order to cope efficiently with condition (\ref{2e2}) we need the following lemma.

\begin{lemma}\label{Lemmaup}
Let $x\in\GD$ be such that $N = N^{a,b}(x) \geq 2$ for some $a < b$. Then there exists $A\in\GA$ such that:
\begin{align} 
\int_0^T x(t)\,dA(t) &\geq (b-a).\\
\|A\|(T) &= 2.\\
\|A\|_{\infty} &= 1/(N-1). 
\end{align}
\end{lemma}
\begin{proof}
Let 
\[0 \leq t_1 < t_2 <t_3 < t_4 < \ldots < t_{2N-1} < t_{2N}\leq T\]
be such that $x(t_{2i-1}) < a,\ x(t_{2i}) > b$, $i=1,2,\ldots, N$. By the right continuity of $x$, for each $=1,2\ldots, N-1$ there are numbers $\delta_i >0$ such that $t_{2i-1} + \delta_i < t_{2i} < t_{2i} +\delta_i < t_{2i+1}$ and 
\[ \sup_{t\in [t_{2i-1},t_{2i-1} + \delta_i]} x(t) \leq a,\quad  \inf_{t\in [t_{2i},t_{2i} + \delta_i]} x(t) \geq b.\]
Let $h_i = 1/\big(\delta_i (N-1)\big)$ and  define 
\[ f(t) = \sum_{i=1}^{N-1} (-1) h_i \GId_{[t_{2i-1}, t_{2i-1} + \delta_i]}(t) +  h_i \GId_{[t_{2i}, t_{2i} + \delta_i]}(t).\]
Then we have for $A = A_f$
\begin{align*}
\int_0^T x(t)\,dA_f(t) &= \int_0^T x(t) f(t)\,dt \\
&= \sum_{i=1}^{N-1} (-1) h_i \int_{t_{2i-1}}^{t_{2i-1} + \delta_i}x(t) \,dt + h_i \int_{t_{2i}}^{t_{2i} + \delta_i} x(t) \,dt \\
&\geq \sum_{i=1}^{N-1} h_i\delta_i \big(b-a\big) =  \big(b-a\big).
\end{align*}
Similarly
\begin{align*}
\|A_f\|(T) &= \int_0^T |f(t)|\,dt =  \sum_{i=1}^{N-1} 2 h_i \cdot \delta_i 
= 2.\\
\|A_f\|_{\infty} &= \sup_{t\in[0,T]} |A_f(t)| = \max_{i=1,2\ldots, N-1} h_i\cdot \delta_i = 1/(N-1).
\end{align*}
\end{proof}

Now suppose that  condition (\ref{2e2}) does not hold for $K = \{x_n\}$. This means that for some $a < b$ and along some subsequence $\{n'\}$ 
\[ N^{a,b}(x_{n'}) \to \infty.\]
Without loss of generality we may assume that for all $ N^{a,b}(x_{n'})  \geq 2$. For each $n'$, let  $A_{n'} = A_{f_{n'}}$ be given by Lemma  \ref{Lemmaup}. Then $A_{n'} \intau A_0 = 0$, while $\int_0^T x_{n'}(t)\,dA_{n'}(t) \geq b-a > 0$ cannot converge to  $\int_0^T x_{0}(t)\,dA_0(t) = 0$. This contradicts (\ref{eq:int}).

\end{proof}

\begin{corollary}
$x_n \starines x_0$ if, and only if,  $x_n(T) \to x_0(T)$ and for each relatively $\tau$-compact set $\cA \subset \GA$
\begin{equation}\label{eq:unif}
 \sup_{A \in \cA} \big| \int_0^T \big(x_n(u)- x_0(u)\big) \, dA(u) \big| \to 0.
\end{equation}
\end{corollary}
\begin{proof} 
Clearly,  (\ref{eq:unif}) implies (\ref{eq:int}). To prove the converse, assume 
(\ref{eq:int}) and suppose that  (\ref{eq:unif}) does not hold for some 
relatively $\tau$-compact set $\cA \subset \GA$. This means that for some $\eta > 0$ there exists a subsequence $\{n'\}$ and elements $A_{n'}$ of
$\cA$ such that for all $n'$
\[ \big| \int_0^T \big(x_{n'}(u)- x_0(u)\big) \, dA_{n'}(u) \big| > \eta.\]
Passing to a $\tau$-convergent subsequence $A_{n^{\bis}}$ we obtain a contradiction with (\ref{eq:int}).
\end{proof}

\begin{remark}
Denote by $\Sigma$ the locally convex topology on $\GD$ given by the seminorm $\rho_1(x) = |x(1)|$ and the seminorms
\[ \rho_{\cA}(x) = \sup_{A\in \cA} \big| \int_0^1 x(u)\,dA(u)\big|,\]
where $\cA$ runs over relatively $\tau$-compact subsets of $\GA$.

 Then $x_n \starines x_0$ if, and only if, $x_n\intopol{\Sigma} x_0$ and 
 so $S \supset \Sigma$, for $S$ is sequential.
\end{remark}

\begin{question}
Is it true that $S \equiv \Sigma$? Positive answer would allow stating that  $(\GD,S)$ is a locally convex linear topological space.
\end{question}

\begin{remark}
Even if $S \varsupsetneq \Sigma$, the compact sets are the same in both topologies, as well as classes of sequentially lower-semicontinuous functions. 
Exploring $S$-compactness and $S$-lower-semicontinuity Guo et al. \cite{GTT17} obtained interesting results on martingale optimal transport on the Skorokhod space. It seems that in problems of such type  the local convexity related to $\Sigma$ can be a useful tool as well.
\end{remark}

\section{$S$ in the hierarchy of topologies}\label{sechier}

Let us begin with listing some facts on  topologies on
$\GD$.

\begin{enumerate}
\item $\GD$ with norm $\|\cdot\|_{\infty}$ is a Banach space, but non-separable.
\item The Skorokhod $J_1$ topology is metric separable and $\big(\GD,J_1\big)$ is topologically complete. For definition and properties of $J_1$ we refer to Billingsley's classic book \cite{Bill68} rather than to its second edition.  
\item It is easy to show that $x_n \intopol{J_1} x_0$ implies $x_n \ines x_0$, hence the $S$ topology is coarser than $J_1$.
\item It was shown in \cite{BJL2016} that the $S$ topology is coarser than Skorokhod's $M_1$ topology (see \cite{Sko56} for definitions of four Skorokhod's topologies).
\item $S$ is incomparable with Skorokhod's $M_2$ topology!
\item $\big(\GD,J_1\big)$ \emph{is not a linear topological space}, for addition is not sequentially $J_1$-continuous, as Figure 1 shows.
\item On the contrary, the sequence $\{f_n\}$ defined in Figure 1 is $S$-convergent to $0$ and exhibits a typical for $S$  phenomenon of self-cancelling oscillations. Addition is \emph{sequentially continuous} in $S$! 
\item We do not know, whether addition is \emph{continuos}, as a function on the product $\GD \times \GD$ with \emph{product topology} $S\times S$ (in general sequential continuity does not imply continuity). Therefore we do not know, whether $(\GD,S)$ is a linear topological space.
\end{enumerate}

\begin{figure}[h]\label{rys1}
\begin{center}\includegraphics[width=10cm]{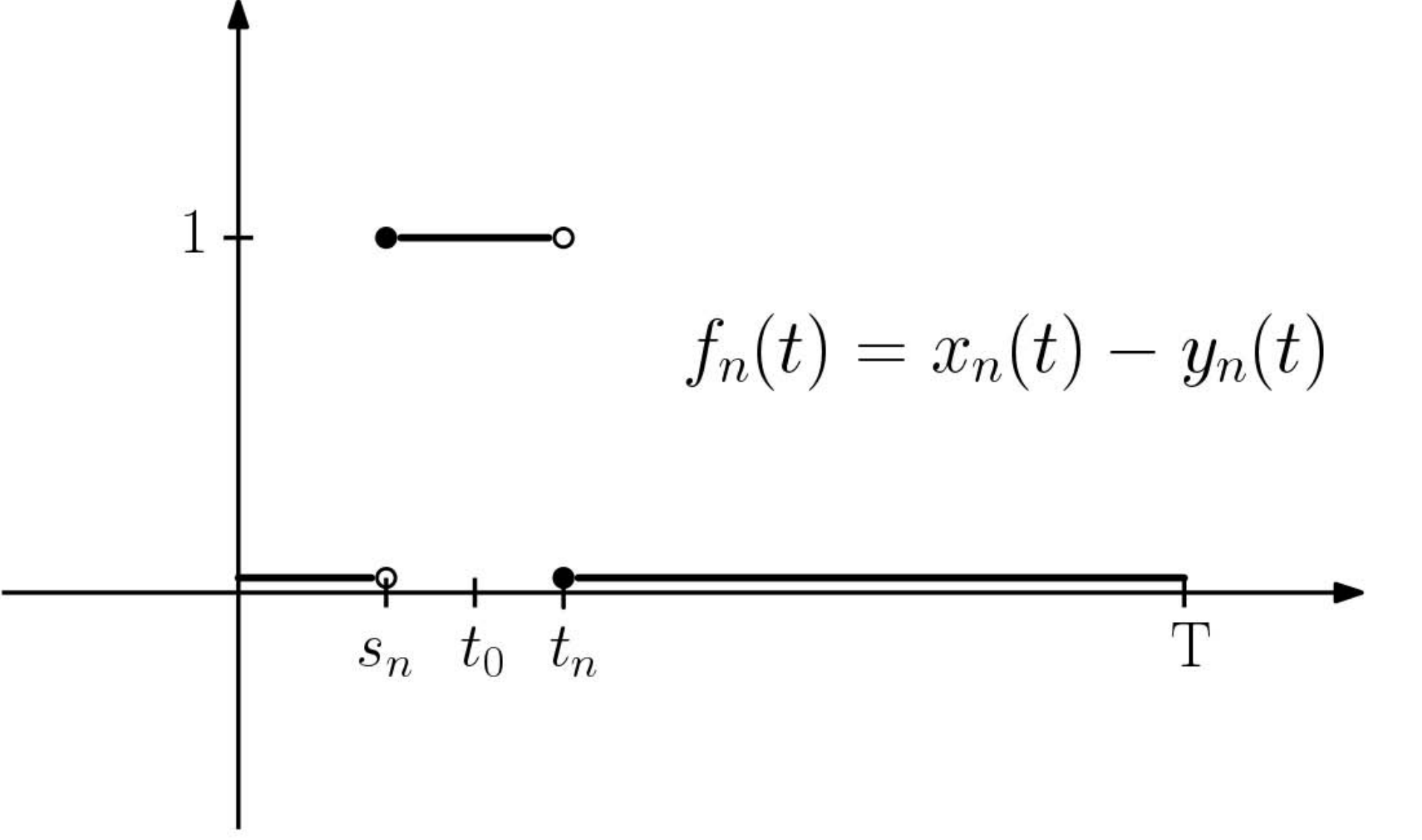}
\end{center}
\caption{$J_1$ is  not linear}
\end{figure}

\begin{theorem}\label{finest}
Suppose $\sigma$ is a topology on $\GD = \GD ([0,T])$ which satisfies the following assumptions.
\begin{align}
&(\GD, \sigma) \text{ is a linear topological space } \label{th4:a}\\
&\sigma \text{ is coarser than the uniform topology generated by the norm } \|\cdot\|_{\infty}. \label{th4:b} \\
&\text{ For each $A > 0$ the set $\{ a \GId_{[u, T]}\,;\, |a| \leq A, u \in [0,T]\}$ is relatively $\sigma$-compact.} \label{th4:c}
\end{align} 

Then $\sigma$ is coarser than the $S$ topology.
\end{theorem}

\begin{proof}
We claim that it is enough to prove that $x_n \ines x_0$ implies $x_{n'} \intopol{\sigma} x_0$ along some subsequence $\{n'\}$. Indeed, this implies that any $\sigma_s$-closed set is also $S$-closed (for $S$ is sequential), and so $\sigma_s$ is coarser than $S$. Since we  always have $\sigma \subset \sigma_s$, our claim follows.

So let us assume that $x_n \ines x_0$. 
 For $\varepsilon
> 0$ and $x \in \GD$ let us define
\begin{align*}
\tau^{\varepsilon}_0(x) &= 0 \\
\tau^{\varepsilon}_k(x) &= \inf \{t > \tau^{\varepsilon}_{k-1}(x) : |x(t)
- x(\tau^{\varepsilon}_{k-1}(x))| > \varepsilon\},\
k=1,2,\ldots.
\end{align*}
(where by convention $\inf \emptyset = +\infty$) and let
\[
v_{\varepsilon}(x)(t) =  x(\tau^{\varepsilon}_k(x))\ \mbox{ if  }
\tau^{\varepsilon}_k(x)\leq t < \tau^{\varepsilon}_{k+1}(x),\
 t\in [0,T], k=0,1,2,\ldots .
\]
Then by the very definition
$
\|x - v_{\varepsilon}(x)\|_{\infty} \leq \varepsilon.
$
Similarly, if we set 
\[ M^{\varepsilon}(x) = \max \{ k\,;\, \tau_k^{\varepsilon}(x) \leq T\},\]
then $M^{\varepsilon}(x) \leq N_{\varepsilon/2}(x)$ and
by Theorem \ref{T:criteria} we have
\begin{align}
\sup_n \|v_{\varepsilon}(x_n)\|_{\infty} &\leq \varepsilon + 
\sup_n \|x_n\|_{\infty} =: A_{\varepsilon} < +\infty, \label{eq:aa}\\
\sup_n M^{\varepsilon}(x_n) &\leq \sup_n N_{\varepsilon/2}(x_n) =: M_{\epsilon} < +\infty.\label{eq:em}
\end{align}
Since $v_{\varepsilon}(x)$ varies only through $M^{\varepsilon}(x)$ jumps, it can be represented as a sum of $M^{\varepsilon}(x)+1$ terms:
\begin{equation}\label{form}
 v_{\varepsilon}(x) = \sum_{k=0}^{M^{\varepsilon}(x)} z_k(x)  \GId_{[\tau_k^{\varepsilon}(x), T]},
\end{equation}
where $z_0(x) = 0$ and
\[ z_k(x) = x\big(\tau_k^{\varepsilon}(x)\big) - \sum_{j=0}^{k-1} x\big(\tau_j^{\varepsilon}(x)\big)\,\quad k =1,2, \ldots, M^{\varepsilon}(x).\]
By (\ref{eq:aa}) and (\ref{eq:em}) we obtain
\[ \sup_n \max_k |z_k(x_n)| \leq (M_{\varepsilon} + 1) A_{\varepsilon}. \]
It follows that the sequence $\{v_{\varepsilon}(x_n)\}$ lives in the algebraic sum
\[ \widetilde{K}_{\varepsilon} = \underbrace{K_{\varepsilon} +  K_{\varepsilon} +  \ldots +  K_{\varepsilon}}_{M_{\varepsilon} + 1\ \text{times}},\]
where $K_{\varepsilon} = \{ a \GId_{[u, T]}\,;\, |a| \leq \big(M_{\varepsilon} + 1\big) A_{\varepsilon}, u \in [0,T]\}$ is relatively $\sigma$-compact. Since $(\GD,\sigma)$ is linear $\widetilde{K}_{\varepsilon}$ is relatively $\sigma$-compact as well. This means that in every subsequence $\{n'\}$ one can find a further subsequence 
$\{n^{\prime\prime}\}$ such that $v_{\varepsilon}(x_{n^{\prime\prime}}) \intopol{\sigma} v_{\varepsilon}$, for some $v_{\varepsilon} \in \GD$. But we can say more: by the special form (\ref{form}) of elements $v_{\varepsilon}(x_n)$ (bounded number of jumps with bounded amplitudes) we may extract a further subsequence $\{n'''\}$ such that  $v_{\varepsilon}(x_{n'''}) \Rightarrow v_{\varepsilon}$.

Now choose $\varepsilon_m \searrow 0$ and apply the diagonal procedure to extract a subsequence $n'$ such that for each $m\in\GN$  we have
along $\{n'\}$
\[ v_{\varepsilon_m}(x_{n'}) \intopol{\sigma} v_m,\qquad  v_{\varepsilon_m}(x_{n'}) \Rightarrow v_m\]
for some $v_m = v_{\varepsilon_m} \in \GD$. Notice that $v_{\varepsilon_m}(x_{n'}) \Rightarrow v_m$ implies $v_{\varepsilon_m}(x_{n'}) \ines v_m$. By Corollary 2.10 in \cite{Jak97} 
\begin{equation}\label{semi}
 \|v_m - x_0\|_{\infty} \leq  \liminf_{n'} \| v_{\varepsilon_m}(x_{n'}) - x_{n'} \|_{\infty} \leq \varepsilon_m \to 0.
\end{equation}
For later purposes we may write $v_m - x_0 \in B_{\varepsilon_m}$, where
 $B_{r} = \{ x\in \GD\,;\, \|x\|_{\infty} \leq r\}$ for $r > 0$.

 Our final task consists in proving $x_{n'} \intopol{\sigma} x_0$.
Let $V$ be a $\sigma$-open neighborhood of $x_0$. By the linearity there exists a $\sigma$-open neighborhood $W$ of $0$ such that $W + W \subset V - x_0$. Since $\sigma$ is coarser than the uniform topology, there exists  $\delta > 0$ such that $B_{2\delta} \subset W$.  Let $m$ be such that $\varepsilon_m < \delta$. Then for $n'$ large enough we have
\begin{align*}
x_{n'}  &= x_{n'} - v_{\varepsilon_m}(x_{n'}) +  v_{\varepsilon_m}(x_{n'}) - v_m + v_m - x_0 + x_0\\
 &\in B_{\varepsilon_m} + W + B_{\varepsilon_m} + x_0 \subset W + W + x_0 \subset V.
\end{align*}
\end{proof}

\begin{remark} Let us consider the space 
$L^p([0,T]) = L^p\big([0,T], \cL|_{[0,T]}, \ell|_{[0,T]}\big)$, 
$p\in [0,+\infty]$ of Lebesgue-measurable functions on $[0,T]$ 
(here $\ell$ stands for the Lebesgue measure). Of course, for each 
$p$ we have $\GD([0,T]) \subset L^p([0,T])$. Moreover, the induced 
metric converts $\GD([0,T])$ into a normed space (if $p\in [1,+\infty)$) 
or a metric linear space (if $p \in [0,1]$). Clearly, assumptions 
(\ref{th4:a}) - (\ref{th4:c}) are satisfied and by our Theorem \ref{finest} 
all mentioned metric topologies are coarser than $S$.
\end{remark}
\begin{remark}
If we replace $\ell$ with another \emph{atomless} finite measure 
$\mu$ on $[0,T]$, then again the metric topologies induced by spaces $L^p([0,T], \mu)$, $p \in [0,+\infty)$,  are coarser than the $S$ topology.

This is not so, if we admit atoms for $\mu$, for (\ref{th4:c}) is then violated.
\end{remark}

\begin{remark}
In Introduction we suggested that  the $S$ topology is \emph{almost} finer than any linear topology which is coarser than Skorokhod's $J_1$ topology. The delicate point is that condition (\ref{th4:c}) does not hold for $J_1$. The corresponding typical example is given in Figure 2, with parameters $t_n \to 0$. To overcome this difficulty we shall introduce a variant of the $J_1$ topology, called $mJ_1$ ($m$ - for modified), which slightly weakens the original topology and for which condition (\ref{th4:c}) is satisfied.
\end{remark}

 \begin{figure}[h]\label{rys2}
\begin{center}\includegraphics[width=10cm]{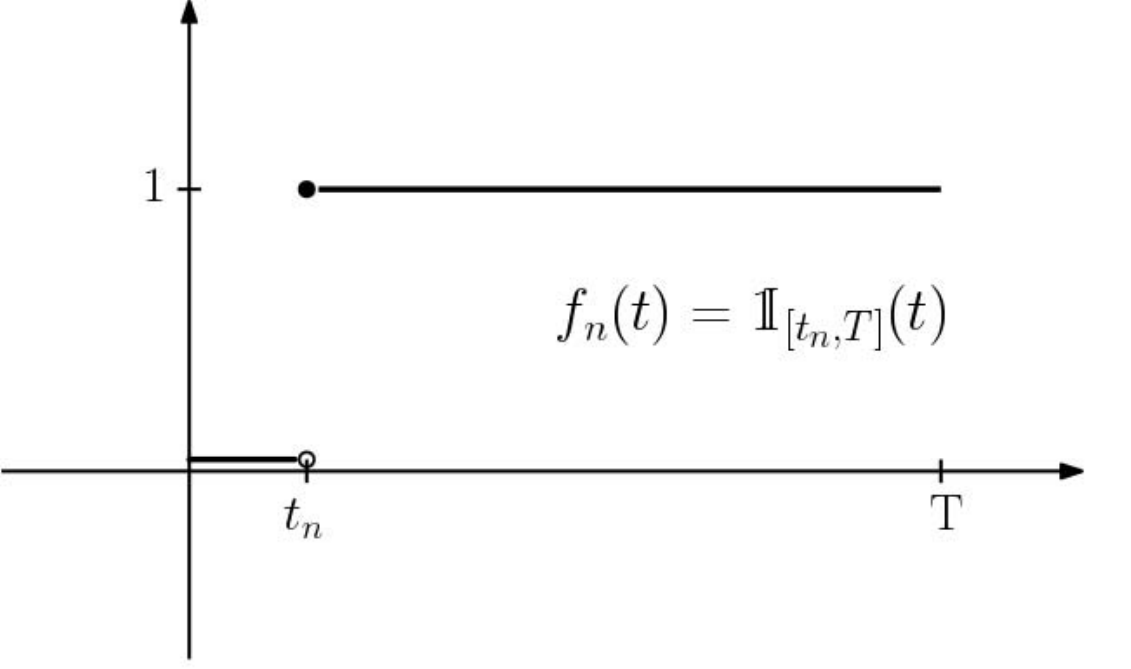}
\end{center}
\caption{A sequence of single jumps that does not converge in $J_1$, but converges in $mJ_1$}
\end{figure}

\begin{definition}[The $mJ_1$ topology]
Fix $\varepsilon > 0$ and consider a one-to-one embedding
\[ \GD([0,T]) \ni x \mapsto  \widetilde{x} \in \GD([-\varepsilon, T + \varepsilon])\]
given by the formula
\begin{equation}
\widetilde{x}(t) = \begin{cases} 0 &\text{ if\quad $-\varepsilon \leq t < 0$},\\
x(t) &\text{ if\quad $0\leq t < T$},\\
x(T) &\text{ if\quad $T \leq t \leq T+\varepsilon$}.
\end{cases}
\end{equation}
Take the complete metric $d_{Sk}$ on $\GD([-\varepsilon, T + \varepsilon])$ (see \cite{Bill68})) and define
\[ d(x,y) = d_{Sk}\big(\widetilde{x},\widetilde{y}\big).\]
Then $(\GD, d)$ becomes a metric space and the corresponding topology will be called $mJ_1$.
\end{definition}

\begin{theorem}[Basic facts on $mJ_1$]\mbox{}

\begin{description}
\item{\bf (i)}  $(\GD, d)$ is a Polish (i.e. complete and separable) metric space.
\item{\bf (ii)} A subset $K \subset \GD$ is relatively $mJ_1$-compact iff it is uniformly bounded:
\begin{equation}
\sup_{x\in K} \sup_{t\in [0,T]} | x(t) | < +\infty,
\end{equation}
and
\begin{equation}
\lim_{\delta \to 0} \sup_{x\in K} \sup_{0- \leq s < t < u \leq T\atop
u - s < \delta} \min\{| x(t) - x(s)|, |x(u) - x(t)|\} = 0,
\end{equation}
where we use the conventions $x(0-) = 0$ and $u - (0-) = u$.
\item{\bf (iii)}  For each $A > 0$ the set $\{ a \GId_{[u, T]}\,;\, |a| \leq A, u \in [0,T]\}$ is relatively $mJ_1$-compact.
\end{description}
\label{emJeyone}
\end{theorem}

\begin{proof}
It is easy to see that if $\widetilde{x}_n$ converges in $\big(\GD([-\varepsilon, T + \varepsilon]), d_{Sk}\big)$ to some $z$, then the limit is of the form
$\widetilde{x}_0$, for some $x_0 \in \GD$. Hence $(\GD, d)$ is homeomorphic 
to the closed subset of the Polish space $\big(\GD([-\varepsilon, T + \varepsilon]), d_{Sk}\big)$, and so it is Polish itself.

Part (ii) is a specification of Theorem 14.4 in \cite{Bill68}.

And part (iii) is a direct consequence of our definition (see Figure 2).

\end{proof}

Taking into account Theorems \ref{finest} and \ref{emJeyone} we obtain
the following interesting result, positioning the $S$ topology in the hierarchy of topologies on $\GD$.

\begin{theorem}[Maximal character of  the $S$ topology]
Every  linear topology on  $\GD$, which is coarser than $ mJ_1$,  is coarser than the $S$ topology as well.
\end{theorem}

\begin{remark}
{Were}   $\big(\GD,S\big)$ a linear topological space,  $S$ would be the finest linear topology on $\GD$ ``below'' $mJ_1$.
\end{remark}

\section{Extensions}\label{extensions}
\subsection{Infinite time horizon}
The problem consists in defining an analog of the $S$ topology on the
 Skorokhod space $\GD\big([0,+\infty)\big)$ of functions 
$x\,:\,\GR^+ \to \GR^1$, which are {right-continuous} at every $t\geq 0$ and
 admit {left limits} at every $t > 0$. This cannot be achieved by invoking
 consistency, because the natural projections of 
$\big(\GD\big([0,T_2]\big), S\big)$   onto $\big(\GD\big([0,T_1]\big), S\big)$, 
$0 < T_1 < T_2$,  are not continuous, due to the special role of the end 
point $T_1 \in (0,T_2)$. 

A similar phenomenon was encountered long time 
ago for Skorokhod's $J_1$ topology (see \cite{Lind73} and \cite{Whit02} 
for the ways to overcome this difficulty).
The case of the $S$ topology can be handled in a somewhat simpler 
manner, mainly due to the characterization of $\starines$ on $\GD$ given in Theorem  \ref{thmain} and the fact that we are interested in convergence of sequences only and not in a particular form of a metric.

\begin{definition}\label{esinfty}
Let $x_n \in \GD\big([0,+\infty)\big)$, $n=0,1,2,\ldots$. We will say that 
$x_n\starines x_0$ in $\GD\big([0,+\infty)\big)$, if for every $T > 0$
\begin{equation}\label{eq:int:er}
\int_0^{T} x_n(t)\, dA_n(t) \to \int_0^{T} x_0(t)\, dA_0(t),
\end{equation}
 for all sequences $\{A_n\} \subset \GA\big([0,T]\big)$ such that $A_n\intau A_0$.

The $S$ topology on $\GD\big([0,+\infty)\big)$ is the sequential topology generated by the  $\cL^*$-convergence  $x_n\starines x_0$.
\end{definition}

If $x\in  \GD\big([0,+\infty)\big)$ and $T > 0$, it will be convenient to denote by $x^T \in \GD\big([0,T]\big)$ the restriction of $x$ to $[0,T]$:
\[ x^T(t) = x(t),\ t\in [0, T].\]

We have the following analog of Theorem \ref{T:criteria}.
\begin{theorem}\label{thcompinfty}
Let $K \subset \GD\big([0,+\infty)\big)$. 

We can find in every sequence $\{x_n\}$ of elements of $K$
 a subsequence $\{x_{n_k}\}$
such that $x_{n_k} \starines x_0$, as $k\to \infty$, if, and only if, 
one of the following equivalent statements {\em (i)} and {\em (ii)} is satisfied.
\begin{description}
\item{\bf (i)} 
\begin{align*}
\sup_{x\in K}  \|x^T\|_{\infty} &< +\infty, \ \text{\em for every $T>0$}.\\
\sup_{x\in K} N^{a,b}(x^T) &< +\infty, \ \text{\em for all $T>0$ and $a < b$}.
\end{align*}
\item{\bf (ii)}
\begin{align*}
\sup_{x\in K}  \|x^T\|_{\infty} &< +\infty, \ \text{\em for every $T>0$}.\\
\sup_{x\in K} N_{\eta}(x^T) &< +\infty, \ \text{\em for all $T>0$ and $\eta > 0$}.
\end{align*}
\end{description}
\end{theorem}

\begin{proof}
The equivalence of (i) and (ii) is stated in Theorem \ref{T:criteria}, so it is enough to deal with (i) only.

\noindent{\em Necessity.}\ \ Suppose that for some $T > 0$ and along 
a sequence $\{x_n\} \subset K$ we have either $\lim_{n\to\infty}
 \sup_{t\in [0,T]} |x_n(t)| = +\infty$ or   
$\lim_{n\to\infty} N^{a,b}\big(x^T_n\big) = +\infty$ for some $a < b$. Choose $T' > T$.
The sequence 
$\{x_n\}$ contains a subsequence $x_{n_k} \starines x_0$ and so
\[ \int_0^{T'} x_{n_k}(t)\, dA_{n_k}(t) \to \int_0^{T'} x_0(t)\, dA_0(t),\ \ k\to+\infty, \]
for all sequences $\{A_{n_k}\} \subset \GA\big([0,T'])$ such that
 $A_{n_k} \intau A_0$ in $\GA\big([0,T'])$. By inspection of the proof of Theorem \ref{thmain} we see that this implies both 
\[ \limsup_{k\to\infty}  \sup_{t\in [0,T]} |x_{n_k}(t)| < +\infty,\]
and
\[  \limsup_{k\to\infty} N_{a,b}(x^T_{n_k}) < +\infty.\]
We have arrived to a contradiction. 

\noindent{\em Sufficiency. }\ \ Take any sequence $T_r \nearrow +\infty$ and assume (i).
By Theorem \ref{T:criteria} we can find a sequence $\{x_{1,n}\}$ such that   $x_{1,n}^{T_1} \ines x_{1,0}$ in  $ \GD\big([0,T_1]\big)$.
In $\{x_{1,n}\}$ we can  find a subsequence $\{x_{2,n}\}$ such that
  $x_{2,n}^{T_2} \ines x_{2,0}$ in  $ \GD\big([0,T_2]\big)$. Repeating this process and then applying the diagonal procedure we can find a subsequence $\{x_{n'}\} \subset K$ such that for every $r\in\GN$
\[ x_{n'}^{T_r} \ines x_{r,0}, \ \text{ in  $ \GD\big([0,T_r]\big)$}.\]
We claim that there exists exactly one $x_0 \in \GD\big([0,+\infty)\big)$ such that for each $r\in\GR$
\[ x_0^{T_r}(t) = x_{r,0}(t),\quad t\in [0,T_r).\]
Let $q < r$. It is enough to verify the consistency of $x_{q,0}$ and $x_{r,0}$ on $[0,T_q)$.
By Corollary 2.9 in \cite{Jak97} we can find a further subsequence $\{n^{\prime\prime}\}$ as well as countable subsets $D_q \subset [0,T_q)$ and   $D_r \subset [0,T_r)$ such that
\[  x_{n^{\prime\prime}}(t) \to x_{q,0}(t),\ t \not\in D_q,\quad x_{n^{\prime\prime}}(t) \to x_{r,0}(t),\ t \not\in D_r.\]  
It follows that if $t$ belongs to the set  $[0,T_q) \setminus \big(D_q \cup D_r)$ that is dense in $[0,T_q)$, then
\[   x_{n^{\prime\prime}}(t) \to x_{q,0}(t) = x_{r,0}(t).\]
Because both $x_{q,0}$ and $x_{r,0}^{T_q}$ are c\`adl\`ag, they are equal
on $[0,T_q)$. 

It remains to show that (\ref{eq:int:er}) holds for $x_{n^{\prime}}$ and $x_0$. Let $T > 0$. Take $T_r \geq T$. By Theorem  \ref{thmain}
\[ \int_0^{T_r} x_{n^{\prime}}(t)\, dA_{n^{\prime}}(t) \to \int_0^{T_r} x_{r,0}(t)\, dA_0(t),\]
 for all sequences $\{A_{n^{\prime}}\} \subset \GA\big([0,T_r]\big)$ such that $A_{n^{\prime}}\intau A_0$. Notice that in view of the continuity of $A_0$, the value of $x_{r,0}$ at $t=T_r$ does not contribute to the value of the integral and therefore the limit integral can be written as
 $\int_0^{T_r} x_{0}(t)\, dA_0(t)$. We have thus established (\ref{eq:int:er}) for $T= T_r$. So let $T < T_r$ and consider a sequence $\{A_{n^{\prime}}\} \subset \GA\big([0,T]\big)$ such that $A_{n^{\prime}}\intau A_0$. Taking a natural extension 
\[\widetilde{A}_{n^{\prime}}(t) = \begin{cases}
A_{n^{\prime}}(t), &\text{ if $t\in [0,T)$}; \\
A_{n^{\prime}}(T), &\text{ if $t\in [T,T_r]$};
\end{cases} \] 
we see that 
\[ \int_0^{T} x_{n^{\prime}}(t)\, dA_{n^{\prime}}(t) = \int_0^{T_r} x_{n^{\prime}}(t)\, d\widetilde{A}_{n^{\prime}}(t), \]
hence the general case is implied by the one already proved.
\end{proof}
\begin{remark} The idea to keep the star over the arrow in the above definition of the convergence generating the $S$ topology on $\GD\big([0,+\infty)\big)$ is justified by the fact that also in the infinite time horizon there  is  a notion corresponding to the $\cL$-convergence $\ines$.
\end{remark}

\begin{theorem} $x_n\starines x_0$ in $\GD\big([0,+\infty)\big)$ if, and only if, in each subsequence $\{x_{n_k}\}$ one can find a further subsequence
 $\{x_{n_{k_l}}\}$ and a sequence $T_r \nearrow +\infty$ such that in each 
$\GD([0,T_r])$
\begin{equation}
\label{elconvinfty}
 x^{T_r}_{n_{k_l}}\ines x_0^{T_r}, \ \text{ as $l\to\infty$}.
\end{equation}
\end{theorem}

\begin{proof} Suppose we have the property described by (\ref{elconvinfty}). $\starines$ is an $\cL^*$-convergence, so it is sufficient to prove that (\ref{elconvinfty}) implies relation (\ref{eq:int:er}). But this is done in the final part of the proof of Theorem \ref{thcompinfty}.

It remains to show that if $x_n \starines x_0$, then we can strengthen the construction given in the sufficiency part of the proof of Theorem \ref{thcompinfty} in such a way that 
\[ x_{r,0}(T_r) = x_0(T_r),\ \ r\in\GR.\]
Let us repeat that construction for some $T_r' \nearrow +\infty$ and find a subsequence $n'$ such that
\[ x_{n'}^{T_r'} \ines x_{r,0}, \ \text{ in  $ \GD\big([0,T_r']\big)$}.\ \ r\in\GN.\]
Because (\ref{eq:int:er}) identifies the limit almost everywhere (as was shown in the proof of Theorem \ref{thmain}), we see that 
$x_{r,0}(t) = x_0(t), \ t\in [0,T_r'), \ \ r\in\GN$. 
Similarly as before, passing to a further subsequence $\{n^{\prime\prime}\}$ 
we have for some countable set $D \subset \GR^+$  
\[ x_{n^{\prime\prime}}(t) \to x_0(t),\ \ t\not\in D.\]
Take any $T_r \in (T_{r-1}', T_r'] \setminus D$. Then $\{x_{n^{\prime\prime}}^{T_r}\}$ is relatively $S$-compact in $\GD\big([0,T_r]\big)$ and 
$x_{n^{\prime\prime}}(t) \to x_0(t)$ on a dense set containing $T_r$. This implies $x_{n^{\prime\prime}} \ines x_0$ in $\GD\big([0,T_r]\big)$ (for if not, we would be able to show that $\{x_{n^{\prime\prime}}^{T_r}\}$ {\em is  not} relatively $S$-compact in $\GD\big([0,T_r]\big)$, just as in the proof of Proposition 2.14 in \cite{Jak97}).
\end{proof}

\begin{definition}
Let $x_n \in \GD\big([0,+\infty)\big)$, $n=0,1,2,\ldots$. We will say that 
$x_n\ines x_0$ in $\GD\big([0,+\infty)\big)$, if one can find a sequence $T_r \nearrow +\infty$ such that for every $r\in\GN$ 
\[x^{T_r}_{n}\ines x_0^{T_r}\ \ \text{in $\GD\big([0,T_r]\big)$,\ \ as $n\to\infty$}.\]
\end{definition}

\subsection{Functions with values in $\GR^d$}

Let $\GD\big([0,+\infty):\GR^d\big)$ be a family of functions $\mathbf{x} : \GR^+ \to \GR^d$, which are right-continuous at every $t \geq 0$ and admit left limits at every $t >0$. Taking coordinates, we may identify  $\GD\big([0,+\infty):\GR^d\big)$ with the product space $\Big(\GD\big([0,+\infty)\big)\Big)^d$. If we equip each space  $\GD\big([0,+\infty)\big)$ with the sequential $S$ topology, then the natural sequential topology on the product is given by the convergence in the components. In other words, we have
\begin{definition}
Let $\mathbf{x}_n \in \GD\big([0,+\infty):\GR^d\big)$, $n=0,1,2,\ldots$, where  
\[\mathbf{x}_n(t) = \big(\mathbf{x}_n^1(t),\mathbf{x}_n^2(t),\ldots, \mathbf{x}_n^d(t)\big).\]
 Then we will say that 
$\mathbf{x}_n\starines \mathbf{x}_0$  if 
\[ \mathbf{x}_n^i \starines \mathbf{x}_n^i\ \ \text{in $\GD\big([0,+\infty)\big)$ for each $i = 1,2,\ldots, d$.}\]
In a similar way we define convergence $\ines$ in $\GD\big([0,+\infty):\GR^d\big)$ and convergences $\starines$ and $\ines$ in $\GD\big([0,T]:\GR^d\big)$. 

The $S$ topology on $\GD\big([0,+\infty):\GR^d\big)$ and on $\GD\big([0,T]:\GR^d\big)$ is the sequential topology generated by $\starines$ considered in the corresponding space. 
\end{definition}

Given this natural definition of  the $S$ topology, the criteria of relative $S$-compactness in the multidimensional setting are obvious.

\begin{theorem}
Let $\mathbf{K} \subset \GD\big([0,+\infty):\GR^d\big)$ and let $\mathbf{K}^i = \{ \mathbf{x}^i ; \mathbf{x} \in \mathbf{K}  \} \subset \GD\big([0,+\infty)\big)$. 

Then $\mathbf{K}$ is relatively $S$-compact in  $\GD\big([0,+\infty):\GR^d\big)$ if, and only if, each set $\mathbf{K}^i$, $i=1,2,\ldots, d$, is relatively compact in $\GD\big([0,+\infty)\big)$.

The same equivalence holds, if we replace  
$\GD\big([0,+\infty):\GR^d\big)$ with $\GD\big([0,T]:\GR^d\big)$ 
and  $\GD\big([0,+\infty)\big)$ with $\GD\big([0,T]\big)$.
\end{theorem}

\section{Appendix: Sequential topologies generated by $\cL$-convergences}\label{appendix}

Following Fr\'echet, we say that $\cX$   is a space of type $\cL$, if among all sequences of elements of $\cX$
a class $\cC(\to)$  of ``convergent" sequences is
distinguished in such a way that:
\begin{description}
\item{(i)}
To each convergent sequence
$(x_n)$  exactly one point $x_0$, called
``the limit",
 is attached (symbolically: $x_n\lr x_0$) 
\item{(ii)}
For every $x\in\cX,$ the constant sequence $(x, x, \ldots)$
is convergent to $x$.
\item{(iii)}
If $x_n\lr x_0$ and $1\leq n_1<n_2<\ldots$, then the
subsequence $(x_{n_k})$  converges, and to the same
limit: $x_{n_k}\lr x_0,$ \ as $k\to\infty.$
\end{description}

Using the $\cL$-convergence $\lr$ one creates the family of closed sets.

\begin{definition}  Say that $F\subset \cX$ is $\tau(\to)$-closed if limits of $\lr$-convergent sequences of elements of $F$ remain in $F$, i.e. if $x_n \in F, \ n\in\GN$ and $x_n\lr x_0$, then $x_0\in F$. The topology given by $\tau(\to)$-closed sets is called the sequential topology generated by the $\cL$-convergence $\lr$ and will be denoted by $\tau(\to)$. 
\end{definition}
\begin{remark}
It must be stressed that for a sequential topology to be
defined only the extremely simple properties (ii) and (iii) of convergence $\lr$
are required. On the other hand, the  topology obtained this way has in general extremely poor separation properties. It is only T1 space due to the fact that in view of (ii) above each one-point set $\{x\}$ is $\tau(\to)$-closed. 

But it is enough for the topology $\tau(\to)$ to define  
a new (in general)
convergence,
``$\conver_{\tau(\to)}$'' say, which, after Urysohn,   is called the convergence
``a posteriori", in order to distinguish
from the original convergence (= convergence ``a priori", i.e. ``$\lr$'').
So $(x_n)$   converges \emph{a posteriori}  to $x_0,$  if
for every $\tau(\to)$-open set $U$ containing $x_0$   eventually all elements of
the sequence $(x_n)$  belong to $U$.  

Kantorovich {\it et al}
\cite[Theorem 2.42, p.51]{Kant50} and Kisy\'nski \cite{Kis60} gave a familiar characterization of the convergence {\em a posteriori} in terms of the convergence {\em a priori}.
\end{remark}

\begin{theorem}[KVPK recipe]\label{kvpk_recipe}
$\{x_n\}$ converges to $x_0$ {\em a posteriori} if, and only if, each subsequence $\{x_{n_k}\}$ contains a further subsequence $\{x_{n_{k_l}}\}$ convergent to $x_0$ {\em a priori}.
\end{theorem}

\begin{remark}
The convergence {\em a posteriori} is generated by a topology. Suppose an $\cL$-convergence $\lr$ satisfies additionally 
\begin{description}
\item{(iv)} If every subsequence $(x_{n_k})$ of $(x_n)$
contains a further subsequence $(x_{n_{k_l}})$
$\lr$-convergent to $x_0,$
then the whole sequence $(x_n)$ is $\lr$-convergent to $x_0$.
\end{description}
Then the convergence $\lr$ is called an 
$\cL^*$-convergence. It is an immediate consequence of Theorem \ref{kvpk_recipe} that if we start with an $\cL^*$-convergence then the convergences {\em a posteriori} and {\em a priori} coincide.
\end{remark}

\begin{remark}
It follows that given an $\cL$-convergence ``$\lr$"  we can {\em weaken}  it to an $\cL^*$-convergence
``$\overset*\lr$" which is already the usual  convergence of
sequences in the topological space
$(\cX, \tau(\lr))\equiv(\cX, \tau(\overset*\to))$.
At least two examples of such a procedure are commonly known.
\end{remark}
\begin{example} If ``$\lr$"  denotes the convergence
``almost surely" of real random variables defined on a
probability space $(\Omega, \cF, \GP),$
then ``$\overset*\lr$" is the convergence ``in probability".
\end{example}
\begin{example} Let $\cX=\GR^1$  and take a sequence
$\varepsilon_n\searrow0.$  Say that $x_n\lr x_0,$  if for
each $n\in\GN, |x_n-x_0|<\varepsilon_n,$   i.e. $x_n$  converges
to $x_0$  at the given rate $\{\varepsilon_n\}.$  Then
``$\overset*\lr$"  means the usual convergence of real
numbers.
\end{example}
\begin{remark} It is worth noting that 
a set $J\subset\cX$ is relatively $\lr$-compact (i.e. in each sequence $\{x_n\} \subset J$ one can find a subsequence $\{n'\}$ such that $x_{n'} \lr x_0$, for some $x_0 \in \cX$) iff
it is relatively $\overset*\lr$-compact. 
\end{remark}

\begin{remark}
Let us notice that if $(\cX, \tau)$  is a
Hausdorff topological space, then
\[\tau\subset\tau_s\equiv\tau(\intau)\]  and in general
this inclusion may be strict (like in the case of the
weak topology on an infinite dimensional Hilbert space).
\end{remark}

\end{document}